\documentclass[11pt]{article}

\usepackage{amsthm,amsmath,amssymb,amsfonts}
\usepackage{mathtools,xfrac}
\usepackage{url,setspace}
\usepackage[usenames,dvipsnames,svgnames,table]{xcolor}
\usepackage{subfigure,graphicx,epsfig,caption}
\usepackage[normalem]{ulem}
\usepackage{geometry}
\usepackage[colorlinks=true]{hyperref}
\usepackage{enumitem,array}
\usepackage{algorithm,algorithmic}
\usepackage{pdfsync}
\usepackage[numbers,sort&compress]{natbib}
\usepackage[in]{fullpage}

\setlist{nosep}

\theoremstyle{plain}
\newtheorem{theorem}{Theorem}

\newtheorem{lemma}[theorem]{Lemma}
\newtheorem*{lemma*}{Lemma}
\newtheorem{proposition}[theorem]{Proposition}
\newtheorem*{proposition*}{Proposition}
\newtheorem{corollary}[theorem]{Corollary}
\newtheorem*{corollary*}{Corollary}

\newtheorem*{observation*}{Observation}

\newtheorem*{example*}{Example}

\theoremstyle{definition}
\newtheorem{definition}{Definition}

\theoremstyle{remark}
\newtheorem{remark}[theorem]{Remark}
\newtheorem*{remark*}{Remark}

\theoremstyle{claim}

\newcommand{\cN}{\mathcal{N}}
\newcommand{\cK}{\mathcal{K}}
\newcommand{\cF}{\mathcal{F}}

\newcommand{\R}{\mathbb{R}}

\newcommand{\bx}{\bar{x}}

\newcommand{\bs}{\bar{s}}

\newcommand{\bF}{\bar{F}}
\newcommand{\bg}{\bar{g}}
\newcommand{\bH}{\bar{H}}
\newcommand{\bK}{\bar{K}}
\newcommand{\bnu}{\bar{\nu}}

\newcommand{\af}{\alpha}

\newcommand{\dx}{\Delta_{\bx}}
\newcommand{\dy}{\Delta_y}
\newcommand{\ds}{\Delta_{\bs}}
\newcommand{\dz}{\Delta_z}

\newcommand{\kx}{k_{\bx}}
\newcommand{\ks}{k_{\bs}}

\newcommand{\la}{\langle}
\newcommand{\ra}{\rangle}

\newcommand{\bse}{\begin{eqnarray*}}
\newcommand{\ese}{\end{eqnarray*}}

\newcommand{\revise}[1]{{#1}}

\begin{document}
\title{On ``A Homogeneous Interior-Point Algorithm for Non-Symmetric Convex Conic Optimization''}
\author{D{\'a}vid Papp\thanks{North Carolina State University, Department of Mathematics. Email: \texttt{dpapp@ncsu.edu}} \
and
Sercan Y{\i}ld{\i}z\thanks{University of North Carolina at Chapel Hill,
Department of Statistics and Operations Research;
Statistical and Applied Mathematical Sciences Institute.}}
\date{\today\footnote{
	This material is based upon work supported by the National Science Foundation under Grant No.~DMS-1719828. Additionally, this material was based upon work partially supported by the National Science Foundation under Grant No.~DMS-1638521 to the Statistical and Applied Mathematical Sciences Institute. Any opinions, findings, and conclusions or recommendations expressed in this material are those of the authors and do not necessarily reflect the views of the National Science Foundation.}}
\maketitle

\abstract{
In a recent paper, Skajaa and Ye proposed a homogeneous primal-dual interior-point method for non-symmetric conic optimization. The authors showed that their algorithm converges to $\varepsilon$-accuracy in $O(\sqrt{\nu}\log \varepsilon^{-1})$ iterations, where $\nu$ is the complexity parameter associated with a barrier function for the primal cone, and thus achieves the best-known iteration complexity for this class of problems. However, an earlier result from the literature was used incorrectly in the proofs of two intermediate lemmas in that paper. In this note, we propose new proofs of these results, allowing the same complexity bound to be established.
}

\medskip
\noindent\textbf{Keywords:} Convex optimization, Non-symmetric conic optimization, Homogeneous self-dual model, Interior-point algorithm

\medskip
\noindent\textbf{Mathematics Subject Classification:} 90C25, 90C51, 90C30

\section{Introduction}
In a recent paper, Skajaa and Ye \cite{SY2015} proposed a homogeneous primal-dual interior-point method for conic optimization problems that have the form
\vspace{-1.5em}
\[
\begin{minipage}[t]{0.45\linewidth}\centering
\begin{alignat*}{3}
\text{(P):}\;   &\min_x\;           &&c^\top x\\
                &\;\text{s.t.}\;    &&Ax=b\\
                &                   &&x\in K
\end{alignat*}
\end{minipage}
\begin{minipage}[t]{0.45\linewidth}\centering
\begin{alignat*}{3}
\text{(D):}\;   &\max_{s,y}\;       &&b^\top y\\
                &\;\text{s.t.}\;    &&A^\top y+s=c\\
                &                   &&s\in K^*,\; y\in\R^m,
\end{alignat*}
\end{minipage}
\]
where $K\subset\R^n$ and $K^*\subset\R^n$ are a dual pair of proper (closed, convex, pointed, and full-dimensional) cones, $A$ is an $m\times n$ real matrix of full row rank, and $b$ and $c$ are real vectors of appropriate dimensions. In contrast to the primal-dual algorithms for non-symmetric conic optimization such as \cite{NTY1999,N2012}, this breakthrough algorithm requires only an efficiently computable barrier function of the primal cone $K$ and assumes nothing about the tractability of barrier functions for $K^*$.

The authors proved that their algorithm converges to $\varepsilon$-accuracy in $O(\sqrt{\nu}\log \varepsilon^{-1})$ iterations, where $\nu$ is the complexity parameter associated with a barrier function for the cone $K$, and thus achieves the best-known iteration complexity for this class of problems. Unfortunately, however, an earlier result from the literature was used incorrectly in the proofs of the two key intermediate results in that paper, Lemmas~5 and 6, and hence the published analysis is incorrect. This corrigendum provides precise versions of these lemmas with corrected proofs. In particular, this shows that the main results of \cite{SY2015} still hold, and allows the same complexity bound to be established. We also suggest other improvements to the previous analysis.

The remainder of this note is organized as follows: In Section~\ref{sec:pre}, we recall the relevant background from the literature and derive some auxiliary results that will be used throughout the paper. In Section~\ref{sec:algo}, we present our new proofs of Lemmas~5 and 6 from \cite{SY2015}. In Appendix~A, we provide the proofs that are omitted from the main text. For completeness, in Appendix~B, we give counterexamples that demonstrate that the original proofs of Lemmas~5 and 6 were indeed incorrect.

\section{Preliminaries}\label{sec:pre}

In this section, we recall important definitions and results about three notions that are central to interior-point method theory: self-concordance, logarithmic homogeneity, and conjugacy. Our presentation is based to a large extent on the excellent textbook \cite{Ren2001}. We also derive some preliminary results that are used in our later analysis.

We introduce our notation before we proceed. The standard inner product on the space $\R^n$ is denoted $\la\cdot,\cdot\ra$; for every $x_1,x_2\in\R^n$, one has $\la x_1,x_2\ra=x_1^\top x_2$. This inner product equips the space $\R^n$ with the norm $\|x\|:=\sqrt{\la x,x\ra}$. 
Throughout this note, $K\subset\R^n$ refers to a proper cone. Its dual cone is $K^*:=\{s\in\R^n:\,\la x,s\ra\geq 0\;\forall x\in K\}$. The notation $C^\circ$ represents the interior of a set $C\subset\R^n$.
Let $F:K^\circ\to\R$ be a twice continuously differentiable function. We use $g(x)$ and $H(x)$ to refer to the gradient and Hessian of $F$ at a point $x\in K^\circ$ with respect to $\la\cdot,\cdot\ra$. We assume that $H(x)$ is positive definite for all $x\in K^\circ$. In particular, $F$ is strictly convex. The Newton step for $F$ at a point $x\in K^\circ$ is defined as $n(x):=-H(x)^{-1}g(x)$.

Every $n\times n$ real symmetric positive definite matrix $S$ gives rise to an inner product $\la\cdot,\cdot\ra_S$ where $\la x_1,x_2\ra_S=\la x_1,Sx_2\ra$ for $x_1,x_2\in\R^n$.
For every $x\in K^\circ$, the inner products $\la\cdot,\cdot\ra_{H(x)}$ and $\la\cdot,\cdot\ra_{H(x)^{-1}}$ equip the space $\R^n$ with the local norms $\|u\|_x:=\|H(x)^{1/2}u\|$ and $\|u\|_x^*:=\|H(x)^{-1/2}u\|$.
These two norms are related through the identity $\|u\|_x^*=\|H(x)^{-1}u\|_x$. Fur\-ther\-more, for every $x,v\in K^\circ$ and $u\in\R^n$, one has $\|u\|_x=\|H(x)^{1/2}u\|=\|H(v)^{-1/2}H(x)^{1/2}u\|_v$. For every $x\in K^\circ$ and $u_1,u_2\in\R^n$, the Cauchy--Schwarz inequality states
\begin{equation}\label{eq:CS}
|\la u_1,u_2\ra_{H(x)}|\leq\|u_1\|_x\|u_2\|_x.
\end{equation}
The norm $\|\cdot\|_x$ induces an operator norm on the space $\R^{n\times n}$ defined as $\|A\|_x:=\max\{\|Au\|_x:\,\|u\|_x\leq 1\}$.
For every $x\in K^\circ$, $u\in\R^n$, and $A\in\R^{n\times n}$, one has
\begin{align}
\|Au\|_x&\leq\|A\|_x\|u\|_x\quad\text{and}\label{eq:lemma1i}\\
\la u,Au\ra_{H(x)}&\overset{\eqref{eq:CS}}{\leq}\|u\|_x\|Au\|_x
\overset{\eqref{eq:lemma1i}}{\leq}\|A\|_x\|u\|_x^2.\label{eq:lemma1ii}
\end{align}


\subsection{Self-Concordance}

In this section, we present useful results about self-concordant functions \cite{NN1994}. We adopt the definition of self-concordance proposed in \cite{Ren2001}. For any $x\in K^\circ$, let $B_x(u,r):=\{v\in\R^n:\,\|v-u\|_x<r\}$ denote the open ball of radius $r>0$ centered at $u\in\R^n$ with respect to the local norm $\|\cdot\|_x$.

\begin{definition}[see Section 2.2.1 in \cite{Ren2001}]
A function $F:K^\circ\to\R$ is said to be \emph{(strongly nondegenerate) self-concordant} if for all $x\in K^\circ$, one has $B_x(x,1)\subset K^\circ$, and for all $v\neq 0$ and $u\in B_x(x,1)$, one has
\begin{equation*}\label{eq:SC}
1-\|u-x\|_x\leq\frac{\|v\|_u}{\|v\|_x}\leq\frac{1}{1-\|u-x\|_x}.
\end{equation*}
\end{definition}

The next two results describe known properties of self-concordant functions.

\begin{theorem}[see Theorem 2.2.1 in \cite{Ren2001}]\label{thm:SCequi}
Assume $F:K^\circ\to\R$ has the property that $B_x(x,1)\subset K^\circ$ for all $x\in K^\circ$. Then $F$ is self-concordant if and only if for all $x\in K^\circ$ and $u\in B_x(x,1)$, one has
\begin{equation}\label{eq:SCequi1}
\|H(x)^{-1}H(u)\|_x,\|H(u)^{-1}H(x)\|_x\leq(1-\|u-x\|_x)^{-2}.
\end{equation}
Likewise, $F$ is self-concordant if and only if for all $x\in K^\circ$ and $u\in B_x(x,1)$, one has
\begin{equation}\label{eq:SCequi2}
\|I-H(x)^{-1}H(u)\|_x,\|I-H(u)^{-1}H(x)\|_x\leq(1-\|u-x\|_x)^{-2}-1.
\end{equation}
\end{theorem}

\begin{theorem}[see Theorem 2.2.4 in \cite{Ren2001}]\label{thm:renegarOrg}
Let $F:K^\circ\to\R$ be a self-concordant function and $x\in K^\circ$ be such that $\|n(x)\|_x<1$. Define $x^+:=x+n(x)$. Then
\begin{equation*}
\|n(x^+)\|_{x^+}\leq\left(\frac{\|n(x)\|_x}{1-\|n(x)\|_x}\right)^2.
\end{equation*}
\end{theorem}

The next theorem generalizes Theorem~\ref{thm:renegarOrg} to the case where the Newton step is damped. Our proof uses the outline of the proof of Theorem 2.2.4 in \cite{Ren2001} and is provided in Appendix~A for completeness.

\begin{theorem}\label{thm:renegarNew}
Let $F:K^\circ\to\R$ be a self-concordant function and $x\in K^\circ$. Choose $x^+:=x+\af n(x)$ such that $0\leq\af\leq 1$ satisfies $\af\|n(x)\|_x<1$. Then
\begin{equation}\label{eq:renegarNew}
\|n(x^+)\|_{x^+}\leq\left(\frac{\af\|n(x)\|_x}{1-\af\|n(x)\|_x}\right)^2+\frac{(1-\af)\|n(x)\|_x}{1-\af\|n(x)\|_x}.
\end{equation}
\end{theorem}

Next we present two lemmas that are used later in Section~\ref{sec:algo}.

\begin{lemma}\label{lem:dNormBaseChg}
Let $F:K^\circ\to\R$ be a self-concordant function. Choose $x\in K^\circ$ and $u\in B_x(x,1)$. Then for all $v\neq 0$, one has
\begin{equation}\label{eq:dNormBaseChg}
\frac{\|v\|_u^*}{\|v\|_x^*}\leq\frac{1}{1-\|u-x\|_x}.
\end{equation}
\end{lemma}
\begin{proof}
Using the change of variables $v=H(x)w$, we get
\[
\max_{v\neq 0}\left(\frac{\|v\|_u^*}{\|v\|_x^*}\right)^2
=\max_{w\neq 0}\left(\frac{\|H(x)w\|_u^*}{\|H(x)w\|_x^*}\right)^2
=\max_{w\neq 0}\frac{\la w,H(u)^{-1}H(x)w\ra_{H(x)}}{\|w\|_x^2}
\overset{\eqref{eq:lemma1ii}}{\leq}\|H(u)^{-1}H(x)\|_x.
\]
The lemma now follows from Theorem~\ref{thm:SCequi}.
\end{proof}

\begin{lemma}\label{lem:gradBnd}
Let $F:K^\circ\to\R$ be a self-concordant function. Choose $x\in K^\circ$ and $u\in B_x(x,1)$. Then
\begin{equation}\label{eq:gradBnd}
\|g(u)-g(x)\|_x^*\leq\frac{\|u-x\|_x}{1-\|u-x\|_x}.
\end{equation}
\end{lemma}
\begin{proof}
Recall from the fundamental theorem of calculus for gradients (see Theorem 1.5.6 in \cite{Ren2001}) that
\[
g(u)-g(x)=\int_0^1 H(x+t(u-x))(u-x)dt.
\]
Using this together with the triangle inequality, we obtain
\begin{align}
\|g(u)-g(x)\|_x^*
&=\left\|\int_0^1 H(x+t(u-x))(u-x)dt\right\|_x^*\notag\\
&\leq\int_0^1\|H(x+t(u-x))(u-x)\|_x^*dt\notag\\
&=\int_0^1\|H(x)^{-1}H(x+t(u-x))(u-x)\|_xdt.\label{eq:gradBndPrf1}
\end{align}
For every $0\leq t\leq 1$, Theorem~\ref{thm:SCequi} indicates
\begin{align}
\|H(x)^{-1}H(x+t(u-x))(u-x)\|_x&\overset{\eqref{eq:lemma1i}}{\leq} \|H(x)^{-1}H(x+t(u-x))\|_x\|u-x\|_x\notag\\
&\overset{\eqref{eq:SCequi1}}{\leq}\frac{\|u-x\|_x}{(1-t\|u-x\|_x)^2}.\label{eq:gradBndPrf2}
\end{align}
Putting \eqref{eq:gradBndPrf1} and \eqref{eq:gradBndPrf2} together, we get
\[
\|g(u)-g(x)\|_x^*
\leq\int_0^1\frac{\|u-x\|_x}{(1-t\|u-x\|_x)^2}dt
=\frac{\|u-x\|_x}{1-\|u-x\|_x}.\qedhere \]
\end{proof}

%

\subsection{Logarithmic Homogeneity}

Interior-point methods for conic optimization make use of a special class of self-concordant functions called logarithmically homogeneous barriers. In this section, we recall an important property of these functions.

\begin{definition}[see Sections 2.3.1 and 2.3.5 in \cite{Ren2001}]
A self-concordant function $F:K^\circ\to\R$ is said to be a \emph{logarithmically homogeneous barrier} if
\begin{enumerate}[label=\roman*.]
\item $\nu:=\sup_{x\in K^\circ}(\|g(x)\|_x^*)^2<\infty$, and
\item $F(tx)=F(x)-\nu\ln t$ for all $x\in K^\circ$ and $t>0$.
\end{enumerate}
The quantity $\nu$ is called the \emph{barrier parameter} of $F$.
\end{definition}


\begin{theorem}[see Theorem 2.3.9 in \cite{Ren2001}]\label{thm:LHBprops}
If $F:K^\circ\to\R$ is a logarithmically homogeneous barrier with parameter $\nu$, then
\begin{equation}\label{eq:LHBprops}
H(x)x=-g(x)\quad\text{and}\quad \|g(x)\|_x^*=\|x\|_x=\sqrt{\la -g(x),x\ra}=\sqrt{\nu}.
\end{equation}
\end{theorem}

\subsection{Conjugacy}

In this section, we recall useful results about conjugates of logarithmically homogeneous barriers.

\begin{definition}[see Section 3.3 in \cite{Ren2001}]
The \emph{conjugate} of the function $F:K^\circ\to\R$ is the function $F^*:(K^*)^\circ\to\R$ defined as
\[
F^*(s):=-\inf_{x\in K^\circ}\{\la x,s\ra+F(s)\}.
\]
\end{definition}


We note that this definition of the conjugate is standard in the literature on interior-point methods for convex optimization, but differs slightly from the classical notion of a Fenchel conjugate in convex analysis. The next two results highlight the duality relationship between $F$ and $F^*$.

\begin{theorem}[see Theorem 3.3.1 in \cite{Ren2001}]\label{thm:SCconj}
If $F:K^\circ\to\R$ is a logarithmically homogeneous barrier with parameter $\nu$, then so is $F^*$.
\end{theorem}

The premise that $F^*$ is a logarithmically homogeneous barrier implies in particular that it is twice continuously differentiable. Let $g^*(s)$ and $H^*(s)$ denote the gradient and Hessian of $F^*$ at a point $s\in(K^*)^\circ$ with respect to $\la\cdot,\cdot\ra$.

\begin{theorem}[see Proposition 3.3.3 and Theorem 3.3.4 in \cite{Ren2001}]\label{thm:SCsymm}
Let $F:K^\circ\to\R$ be a logarithmically homogeneous barrier.
\begin{enumerate}[label=\roman*.]
\item The gradient map $g:K^\circ\to\R^n$ defines a bijection between $K^\circ$ and $-(K^*)^\circ$.
\item If $x\in K^\circ$ and $s\in(K^*)^\circ$ satisfy $s=-g(x)$, then
\begin{equation}\label{eq:SCsymm}
-g^*(s)=x\quad\text{and}\quad H^*(s)=H(x)^{-1}.
\end{equation}
\end{enumerate}
\end{theorem}

For any $s\in(K^*)^\circ$, the inner product $\la\cdot,\cdot\ra_{H^*(s)}$ equips the space $\R^n$ with the norm $\|u\|_s^*:=\|H^*(s)^{1/2}u\|$.
Theorem~\ref{thm:SCsymm}(i) indicates that there exists $x\in K^\circ$ such that $s=-g(x)$, and Theorem~\ref{thm:SCsymm}(ii) implies
\begin{equation}\label{eq:dNorms}
\|u\|_s^*=\|H^*(s)^{1/2}u\|\overset{\eqref{eq:SCsymm}}{=}\|H(x)^{-1/2}u\|=\|u\|_x^*.
\end{equation}

\section{The Skajaa--Ye Algorithm}\label{sec:algo}

In the paper \cite{SY2015}, Skajaa and Ye present an interior-point method that finds a feasible solution to the homogeneous self-dual embedding of problems (P) and (D):
\begin{gather}\label{eq:HSD}
\begin{aligned}
&               &&\phantom{--}Ax  &&-b\tau&&          &&          &&=0\\
&-A^\top y      &&          &&+c\tau&&-s        &&          &&=0\\
&\phantom{--}b^\top y &&-c^\top x &&      &&          &&-\kappa   &&=0
\end{aligned}\\
\begin{gathered}
y\in\R^m,\quad (x;\tau)\in K\times\R_+,\quad (s;\kappa)\in K^*\times\R_+.\notag
\end{gathered}
\end{gather}
The authors show that, if (P) and (D) are both feasible and have zero duality gap, Algorithm~\ref{alg:SY} returns a solution that can be converted into optimal solutions to the original problems (P) and (D), and if one or both of (P) and (D) are infeasible, Algorithm~\ref{alg:SY} returns infeasibility certificates for those problems (see Lemma 1 in \cite{SY2015} and the discussion that follows).

Before describing the Skajaa--Ye algorithm, we present some of the notation from \cite{SY2015}. For convenience, let
\begin{alignat*}{3}
&\bx:=(x;\tau),\qquad&&\bs:=(s;\kappa),\qquad&&z:=(\bx;y;\bs),\\
&\bK:=K\times\R_+,\qquad&&\bK^*:=K^*\times\R_+,\qquad&&\cF:=\bK\times\R^m\times\bK^*.
\end{alignat*}
Define $G$ as the skew-symmetric matrix
\[
G:=
\left(
\begin{array}{ccc}
 0      & A     & -b\\
-A^\top & 0     & \phantom{-}c\\
\phantom{-}b^\top &-c^\top& \phantom{-}0
\end{array}
\right).
\]
In this notation, the model \eqref{eq:HSD} can be expressed compactly as
\[
G(y;\bx)-(0;\bs)=(0;0)\quad\text{and}\quad z=(\bx;y;\bs)\in\cF.
\]

Let $F:K^\circ\to\R$ be a logarithmically homogeneous barrier for $K$ with parameter $\nu$.
Recall from Theorem~\ref{thm:SCconj} that its conjugate $F^*:(K^*)^\circ\to\R$ is also a logarithmically homogeneous barrier with parameter $\nu$. Note that $\nu\geq 1$ because $K$ is pointed (see Corollary~2.3.3 in \cite{NN1994}).
We define the functions $\bF:\bK^\circ\to\R$ and $\bF^*:(\bK^*)^\circ\to\R$ as
\[
\bF(\bx) := F(x)-\log\tau\quad\text{and}\quad\bF^*(\bs) := F^*(s)-\log\kappa.
\]
These are logarithmically homogeneous barriers with parameter $\bnu:=\nu+1\geq 2$. Here we diverge slightly from the notation of \cite{SY2015} to explicitly distinguish between the derivatives of $F$ and $\bF$: We let $\bg(\bx)$ and $\bH(\bx)$ denote the gradient and Hessian of $\bF$ at a point $\bx\in\bK^\circ$ with respect to the standard inner product $\la\cdot,\cdot\ra$ in $\R^{n+1}$. Similarly, we let $\bg^*(\bs)$ and $\bH^*(\bs)$ denote the gradient and Hessian of $\bF^*$ at a point $\bs\in(\bK^*)^\circ$ with respect to the inner product $\la\cdot,\cdot\ra$. For every $\bx\in\bK^\circ$ and $\bs\in(\bK^*)^\circ$, we define the local norms $\|u\|_{\bx}:=\|\bH(\bx)^{1/2}u\|$,  $\|u\|_{\bx}^*:=\|\bH(\bx)^{-1/2}u\|$, and $\|u\|_{\bs}^*:=\|\bH^*(\bs)^{1/2}u\|$. As mentioned in Section~\ref{sec:pre}, these local norms are related through the identities $\|\bH(\bx)^{-1}u\|_{\bx}=\|u\|_{\bx}^*$ and $\|u\|_{\bx}^*=\|u\|_{\bs}^*$ where $\bs=-\bg(\bx)$.

Let $\mu(z):=\bx^\top\bs/\bnu$ denote the \emph{complementarity gap} of $z=(\bx;y;\bs)\in\cF$. Let $\psi:\bK^\circ\times(\bK^*)^\circ\times\R\to\R^{n+1}$ be the map defined as
\[
\psi(\bx,\bs,t):=\bs+t\bg(\bx).
\]
It follows directly from these definitions and Theorem~\ref{thm:LHBprops} that, for every $z=(\bx;y;\bs)\in\cF$, one has
\begin{equation}\label{eq:psix}
\psi(\bx,\bs,\mu(z))^\top\bx
=\bs^\top\bx+\mu(z)\bg(\bx)^\top\bx\overset{\eqref{eq:LHBprops}}{=}\mu(z)\bnu-\mu(z)\bnu=0.
\end{equation}
We refer the reader to Section 4.2 of \cite{SY2015} for a formal description of the central path for the model \eqref{eq:HSD}.
For fixed $\theta\in[0,1]$, the $\theta$-neighborhood of the central path is defined as
\[
\cN(\theta):=\left\{z=(\bx;y;\bs)\in\cF^\circ:\;\|\psi(\bx,\bs,\mu(z))\|_{\bx}^*\leq\theta\mu(z)\right\}.
\]


We state the Skajaa--Ye algorithm \cite{SY2015} as Algorithm~\ref{alg:SY} below. The algorithm alternates between a predictor phase and a corrector phase until certain termination criteria are satisfied. Each corrector phase consists of $r_c$ consecutive corrector steps for some fixed parameter $r_c>0$. At each predictor and corrector step, the update direction is computed solving a linear system, and the current solution is updated along this direction using the step length $\af_p>0$ in a predictor step and the step length $\af_c>0$ in a corrector step. With appropriately fixed parameters $\beta>\eta>0$, the algorithm maintains the invariants that the predictor step updates a solution $z\in\cN(\eta)$ to a solution $z^+\in\cN(\beta)$ and the sequence of $r_c$ corrector steps update a solution $z\in\cN(\beta)$ to a solution $z^+\in\cN(\eta)$.

The complexity analysis shows that the parameters $\eta,\beta,\af_p,\af_c$ and $r_c$ can be chosen such that, given any initial solution $z^0=(\bx^0;y^0;\bs^0)\in\cN(\eta)$ and $\varepsilon>0$, the algorithm converges to a solution $z=(\bx;y;\bs)\in\cF^\circ$ that satisfies
\[
\mu(z)\leq\varepsilon\mu(z^0)\quad\text{and}\quad\|G(y;\bx)-(0;\bs)\|\leq\varepsilon\|G(y^0;\bx^0)-(0;\bs^0)\|
\]
in $O(\sqrt{\nu}\log\varepsilon^{-1})$ iterations (see Theorem~1 in \cite{SY2015}). In the algorithm, as stated, the step sizes are fixed both in the predictor and in the corrector steps. The analysis is also applicable to the variant that instead uses line search in the predictor phase to find the (approximately) largest $\af_p$ for which $z+\af_p\dz \in \cN(\beta)$.

{\centering
\begin{minipage}[T]{\linewidth}
\begin{algorithm}[H]
\caption{Predictor-Corrector Algorithm for Non-Symmetric Cone Optimization}\label{alg:SY}
\begin{algorithmic}
\STATE \hspace{-1em}\textbf{Parameters:} Real numbers $\eta,\beta,\af_p,\af_c>0$ and integer $r_c>0$
\STATE \hspace{-1em}\textbf{Input:} Logarithmically homogeneous barrier $F$ for $K$ and initial point $z\in\cN(\eta)$
\LOOP
\STATE \textbf{Termination?}
\STATE If termination criteria are satisfied, stop and return $z$.
\STATE \textbf{Prediction}
\STATE Solve the linear system \eqref{eq:eqnSysPred} for $\dz=(\dx;\dy;\ds)$.
\STATE Set $z\leftarrow z+\af_p\dz$.
\STATE \textbf{Correction}
\FORALL{$i=1,\ldots,r_c$}
\STATE Solve the linear system \eqref{eq:eqnSysCorr} for $\dz=(\dx;\dy;\ds)$.
\STATE Set $z\leftarrow z+\af_c\dz$.
\ENDFOR
\ENDLOOP
\end{algorithmic}
\end{algorithm}
\end{minipage}
\vspace{1em}
}

Two key intermediate results in this complexity analysis are Lemmas~5 and 6 in \cite{SY2015}, which demonstrate that the updated solution at the end of the predictor phase (resp. corrector phase) satisfies $z^+\in\cN(\beta)$ (resp. $z^+\in\cN(\eta)$) with the suggested parameters. However, Theorem~\ref{thm:renegarOrg} stated above was used incorrectly in the proofs of these results; we demonstrate this with counterexamples in Appendix~B. We propose Lemmas~\ref{lem:newLemma5} and \ref{lem:newLemma6} below to replace the two aforementioned lemmas from \cite{SY2015}. Our proof of Lemma~\ref{lem:newLemma6} uses the more general Theorem~\ref{thm:renegarNew} instead of Theorem~\ref{thm:renegarOrg}, while our proof of Lemma~\ref{lem:newLemma5} does not make use of Theorem~\ref{thm:renegarOrg} at all. Furthermore, whereas Skajaa and Ye \cite{SY2015} state and analyze their algorithm for the case $r_c=2$, our analysis shows that the same asymptotic complexity result can be established for any fixed $r_c>0$. Therefore, in the statements of Lemmas~\ref{lem:newLemma5} and \ref{lem:newLemma6} below, we provide suitable parameter values for both $r_c=1$ and $r_c=2$.

\begin{lemma}\label{lem:newLemma5}
Let $\kx:=\eta+\sqrt{2\eta^2+\bnu}$. Assume one of the following:
\begin{itemize}
\item $\beta=0.20$, $\epsilon=0.50$, $r_c=1$, $\eta=\beta\epsilon^{r_c}$, and $\af_p=0.020\kx^{-1}$, or
\item $\beta=0.25$, $\epsilon=0.70$, $r_c=2$, $\eta=\beta\epsilon^{r_c}$, and $\af_p=0.025\kx^{-1}$.
\end{itemize}
If $z\in\cN(\eta)$, then the predictor phase yields a solution $z^+$ that satisfies $z^+\in\cN(\beta)$.
\end{lemma}

\begin{lemma}\label{lem:newLemma6}
Assume one of the following:
\begin{itemize}
\item $\beta=0.20$, $\epsilon=0.50$, $r_c=1$, $\eta=\beta\epsilon^{r_c}$, and $\af_c=1$, or
\item $\beta=0.25$, $\epsilon=0.70$, $r_c=2$, $\eta=\beta\epsilon^{r_c}$, and $\af_c=1$.
\end{itemize}
If $z\in\cN(\beta)$, then the corrector phase yields a solution $z^+$ that satisfies $z^+\in\cN(\eta)$.
\end{lemma}
The next two sections are dedicated to the proofs of Lemmas~\ref{lem:newLemma5} and \ref{lem:newLemma6}.

\subsection{Prediction}

In this section, we present the proof of Lemma~\ref{lem:newLemma5}. In the paper \cite{SY2015}, the update direction $(\dx;\dy;\ds)$ at each predictor step is computed as the solution to the linear system
\begin{subequations}\label{eq:eqnSysPred}
\begin{alignat}{1}
G(\dy;\dx)-(0;\ds)&=-(G(y;\bx)-(0;\bs)),\label{eq:eqnSysPred1}\\
\ds+\mu(z)\bH(\bx)\dx&=-\bs.\label{eq:eqnSysPred2}
\end{alignat}
\end{subequations}
The next three lemmas summarize useful results from \cite{SY2015} about the predictor step.

\begin{lemma}[see (34-35) in \cite{SY2015}]\label{lem:dxdsBndPred}
Let $z=(\bx;y;\bs)\in\cN(\eta)$ for some $0\leq\eta\leq 1$, and let $\dz=(\dx;\dy;\ds)$ be the solution to \eqref{eq:eqnSysPred}. Then
\begin{equation}\label{eq:dxdsBndPred}
\|\dx\|_{\bx}\leq\kx\quad\text{and}\quad\|\ds\|_{\bx}^*\leq\ks\mu(z),
\end{equation}
where $\kx:=\eta+\sqrt{2\eta^2+\bnu}$ and $\ks:=\kx+\sqrt{\kx^2+\eta^2+\bnu}$.
\end{lemma}

\revise{
\begin{lemma}[see Lemma~3 in \cite{SY2015}]\label{lem:resRedPred}
Let $z=(\bx;y;\bs)\in\cK^\circ$, and let $\dz=(\dx;\dy;\ds)$ be the solution to \eqref{eq:eqnSysPred}. Choose $z^+=(\bx^+;y^+;\bs^+)=z+\af_p\dz$ for some scalar $\af_p$. Then
\begin{align}
G(y^+;\bx^+)-(0;\bs^+)&=(1-\af_p)(G(y;\bx)-(0;\bs))\qquad\text{ and }\label{eq:infRedPred}\\
\mu(z^+)&=(1-\af_p)(\mu(z)+\af_p\bnu^{-1}\psi^\top\dx).\label{eq:muRedPred}
\end{align}
\end{lemma}
}

\begin{lemma}[see (38-39) in \cite{SY2015}]\label{lem:muChgPred}
Let $z=(\bx;y;\bs)\in\cN(\eta)$ for some $0\leq\eta\leq 1$, and let $\dz=(\dx;\dy;\ds)$ be the solution to \eqref{eq:eqnSysPred}. Choose $z^+=(\bx^+;y^+;\bs^+):=z+\af_p\dz$ for some $0\leq\af_p<\kx^{-1}$. Then
\begin{gather}
|\mu(z^+)-\mu(z)|\leq\mu(z)\af_p\left(1+(1-\af_p)\eta\kx\bnu^{-1}\right),\label{eq:muDiffPred}\\
\revise{\left(1-\af_p\right)\left(1-\af_p\eta\kx\bnu^{-1}\right)\leq\frac{\mu(z^+)}{\mu(z)}\leq\left(1-\af_p\right)\left(1+\af_p\eta\kx\bnu^{-1}\right).}\label{eq:muRatioPred}
\end{gather}
\end{lemma}

It is immediate from the definitions of $\kx$ and $\ks$ that $2\kx\leq\ks$. Our next remark provides a simple upper bound on $\ks$ in terms of $\kx$.

\begin{remark}
Suppose $\kx$ and $\ks$ are defined as in Lemma~\ref{lem:dxdsBndPred}. Then
\begin{equation}\label{eq:ksBnd}
\ks\leq(1+\sqrt{2})\kx.
\end{equation}
\end{remark}
\begin{proof}
The definition of $\kx$ implies $\kx^2\geq\eta^2+\bnu$. Then $\ks^2=\kx^2+2\kx\sqrt{\eta^2+\bnu+\kx^2}+\eta^2+\bnu+\kx^2\leq(3+2\sqrt{2})\kx^2$.
Taking the square roots of both sides produces the claimed result.
\end{proof}


The next lemma shows that the updated solution after a predictor step belongs to $\cF^\circ$ if the step size is small enough. The first part of the lemma is from \cite{SY2015}, whereas the second part presents a slight strengthening of a similar result from \cite{SY2015} which requires $\af_p<(1-\eta)\ks^{-1}$ (see the section ``Feasibility of $z^+$" of Appendix~3 in \cite{SY2015}).

\begin{lemma}\label{lem:inPred}
Let $z=(\bx;y;\bs)\in\cN(\eta)$ for some $0\leq\eta\leq 1$, and let $\dz=(\dx;\dy;\ds)$ be the solution to \eqref{eq:eqnSysPred}. Choose $z^+=(\bx^+;y^+;\bs^+):=z+\af_p\dz$ for some $0\leq\af_p<\kx^{-1}$. Then the following statements are true:
\begin{enumerate}[label=\roman*.]
\item $\bx^+\in\bK^\circ$.
\item If $\|\psi(z^+)\|_{\bx^+}^*\leq\beta\mu(z^+)$ for some $0\leq\beta<1$, then $\bs^+\in(\bK^*)^\circ$.
\end{enumerate}
\end{lemma}
\begin{proof}
The proof of claim~(i) can be found in Appendix~3 of \cite{SY2015}. For claim~(ii), let $\mu:=\mu(z)$, $\mu^+:=\mu(z^+)$, and $\psi^+:=\psi(\bx^+,\bs^+,\mu^+)=\bs^+ +\mu^+\bg(\bx^+)$.
Note that $\mu>0$ because $\bx\in\bK^\circ$ and $\bs\in(\bK^*)^\circ$. We claim $\mu^+>0$. To see this, note first that $0\leq\eta\bnu^{-1}<1$ for $\bnu\geq 2$ and $0\leq\eta\leq 1$. Combining this inequality with $0\leq\af_p\kx<1$, we get $0\leq\af_p\kx\eta\bnu^{-1}<1$. Furthermore, we have $\af_p<\kx^{-1}<1$ because $\kx\geq\sqrt{\bnu}>1$. Then
\[
\af_p\left(1+(1-\af_p)\kx\eta\bnu^{-1}\right)=\af_p+(1-\af_p)\af_p\kx\eta\bnu^{-1}<\af_p+(1-\af_p)=1.
\]
Together with \eqref{eq:muDiffPred}, this inequality gives $\mu-\mu^+<\mu$, which implies the desired $\mu^+>0$. From Theorem~\ref{thm:SCsymm}(i), recall that $\bx^+\in\bK^\circ$ implies $-\bg(\bx^+)\in(\bK^*)^\circ$. Using $\mu^+>0$, we can write
\[
\|\bs^+/\mu^+ + \bg(\bx^+)\|_{-\bg(\bx^+)}^*\overset{\eqref{eq:dNorms}}{=}\|\bs^+/\mu^+ + \bg(\bx^+)\|_{\bx^+}^* =(\mu^+)^{-1}\|\psi^+\|_{\bx^+}^*\leq\beta<1.
\]
The conclusion $\bs^+\in(\bK^*)^\circ$ now follows from the last inequality together with $-\bg(\bx^+)\in(\bK^*)^\circ$ and the self-concordance of $\bF^*$.
\end{proof}


The next proposition is the main result of this section.

\begin{proposition}\label{prop:psiBndPred}
Let $z=(\bx;y;\bs)\in\cN(\eta)$ for some $0\leq\eta\leq 1$, and let $\dz=(\dx;\dy;\ds)$ be the solution to \eqref{eq:eqnSysPred}. Choose $z^+=(\bx^+;y^+;\bs^+):=z+\af_p\dz$ for some $\af_p>0$ such that $c_p:=\af_p\kx<1$. Then
\begin{equation}\label{eq:psiBndPred}
\mu(z^+)^{-1}\|\psi(z^+)\|_{\bx^+}^*\leq\frac{c_p}{(1-c_p)^2}+\frac{2\eta(\sqrt{2}+c_p)+4(1+\sqrt{2})c_p}{(1-c_p)(\sqrt{2}-c_p)(2-c_p\eta)}.
\end{equation}
\end{proposition}
\begin{proof}
Let $\mu:=\mu(z)$, $\mu^+:=\mu(z^+)$, $\psi:=\psi(\bx,\bs,\mu)=\bs +\mu\bg(\bx)$, and $\psi^+:=\psi(\bx^+,\bs^+,\mu^+)=\bs^+ +\mu^+\bg(\bx^+)$. Lemma~\ref{lem:inPred}(i) shows that $\bx^+\in\bK^\circ$ under the hypotheses of this proposition. Furthermore, $z\in\cN(\eta)$ implies $\|\psi\|_{\bx}^*\leq\eta\mu$. We use Theorem~\ref{thm:LHBprops} and Lemma~\ref{lem:dxdsBndPred} together with the triangle inequality to obtain
\begin{align*}
\|\psi^+\|_{\bx}^*
&=\|\psi +\mu^+(\bg(\bx^+)-\bg(\bx))+(\mu^+-\mu)\bg(\bx)+(\bs^+-\bs)\|_{\bx}^*\\
&\leq\|\psi\|_{\bx}^*+\mu^+\|\bg(\bx^+)-\bg(\bx)\|_{\bx}^*+|\mu^+-\mu|\|\bg(\bx)\|_{\bx}^*+\|\bs^+-\bs\|_{\bx}^*\\
&\hspace{-0.75em}\overset{(\ref{eq:LHBprops},\ref{eq:dxdsBndPred})}{\leq}
\eta\mu+\mu^+\|\bg(\bx^+)-\bg(\bx)\|_{\bx}^*+|\mu^+-\mu|\sqrt{\bnu}+\af_p\ks\mu.
\end{align*}
Using Lemmas~\ref{lem:gradBnd} and \ref{lem:muChgPred} to bound the right-hand side of this chain yields
\begin{align*}
\|\psi^+\|_{\bx}^*
&\hspace{-0.25em}\overset{(\ref{eq:gradBnd},\ref{eq:muDiffPred})}{\leq}\frac{\af_p\kx\mu^+}{1-\af_p\kx}+\mu\left(\eta +\af_p\ks+\af_p(\sqrt{\bnu}+(1-\af_p)\eta\kx\bnu^{-1/2})\right)\\
&\overset{\eqref{eq:muRatioPred}}{\leq}\frac{\af_p\kx\mu^+}{1-\af_p\kx}+\frac{\mu^+\left(\eta +\af_p\ks+\af_p\sqrt{\bnu}+\af_p(1-\af_p)\eta\kx\bnu^{-1/2}\right)}{(1-\af_p)(1-\af_p\eta\kx\bnu^{-1})}.
\end{align*}
Note that $0\leq\af_p\eta\kx<1$ because $0\leq\eta\leq 1$ and $0<\af_p\kx<1$. We can use the bounds $\af_p>0$, $\sqrt{\bnu}\leq\kx$, and $\ks\leq(1+\sqrt{2})\kx$ (see \eqref{eq:ksBnd}) to obtain
\begin{align}
(1/\mu^+)\|\psi^+\|_{\bx}^*
&\leq\frac{\af_p\kx}{1-\af_p\kx}+\frac{\eta +(2+\sqrt{2})\af_p\kx+\af_p\eta\kx\bnu^{-1/2}}{(1-\af_p\kx\bnu^{-1/2})(1-\af_p\eta\kx\bnu^{-1})}\notag\\
&=\frac{c_p}{1-c_p}+\frac{\eta+(2+\sqrt{2})c_p+c_p\eta\bnu^{-1/2}}{(1-c_p\bnu^{-1/2})(1-c_p\eta\bnu^{-1})}.\label{eq:psiBndPredPrf}
\end{align}
Now recall from Lemmas~\ref{lem:dNormBaseChg} and \ref{lem:dxdsBndPred} that
\[
\|\psi^+\|_{\bx^+}^*
\overset{\eqref{eq:dNormBaseChg}}{\leq}(1-\af_p\|\dx\|)^{-1}\|\psi^+\|_{\bx}^*
\overset{\eqref{eq:dxdsBndPred}}{\leq}(1-\af_p\kx)^{-1}\|\psi^+\|_{\bx}^*
=(1-c_p)^{-1}\|\psi^+\|_{\bx}^*.
\]
Combining this inequality with \eqref{eq:psiBndPredPrf} produces
\[
(\mu^+)^{-1}\|\psi^+\|_{\bx^+}^*
\leq\frac{c_p}{(1-c_p)^2}+\frac{\eta(1+c_p\bnu^{-1/2})+(2+\sqrt{2})c_p}{(1-c_p)(1-c_p\bnu^{-1/2})(1-c_p\eta\bnu^{-1})}.
\]
Recall that $\bnu\geq 2$. Observing that the right-hand side of this inequality is monotone nonincreasing in $\bnu$ for $\bnu\geq 2$ and replacing $\bnu$ with $2$ yields \eqref{eq:psiBndPred}.
\end{proof}

Our next corollary implies Lemma~\ref{lem:newLemma5} stated at the beginning of Section~\ref{sec:algo}.
\begin{corollary}
Assume one of the following:
\begin{itemize}
\item $\beta=0.20$, $\epsilon=0.50$, $r_c=1$, $\eta=\beta\epsilon^{r_c}$, and $\af_p=0.020\kx^{-1}$, or
\item $\beta=0.25$, $\epsilon=0.70$, $r_c=2$, $\eta=\beta\epsilon^{r_c}$, and $\af_p=0.025\kx^{-1}$.
\end{itemize}
Let $z=(\bx;y;\bs)\in\cN(\eta)$, and let $\dz=(\dx;\dy;\ds)$ be the solution to \eqref{eq:eqnSysPred}. Then $z^+=(\bx^+;y^+;\bs^+):=z+\af_p\dz\in\cN(\beta)$.
\end{corollary}
\begin{proof}
It can be verified from Lemma~\ref{lem:inPred}(i) and Proposition~\ref{prop:psiBndPred} that $\bx^+\in\bK^\circ$ and $\|\psi(\bx^+,\bs^+,\mu(z^+))\|_{\bx^+}^*\allowbreak\leq\beta\mu(z^+)$ with the prescribed parameters. Because $\beta<1$, Lemma~\ref{lem:inPred}(ii) further implies that $\bs^+\in(\bK^*)^\circ$.
\end{proof}

\subsection{Correction}

In this section, we present the proof for Lemma~\ref{lem:newLemma6}. Recall that the corrector phase of Algorithm~\ref{alg:SY} consists of $r_c$ successive corrector steps. In the paper \cite{SY2015}, the corrector update direction $(\dx;\dy;\ds)$ is computed at each of these steps as the solution to the linear system
\begin{subequations}\label{eq:eqnSysCorr}
\begin{alignat}{1}
G(\dy;\dx)-(0;\ds)&=0,\label{eq:eqnSysCorr1}\\
\ds+\mu(z)\bH(\bx)\dx&=-\psi(z).\label{eq:eqnSysCorr2}
\end{alignat}
\end{subequations}
Since $G$ is skew-symmetric, multiplying \eqref{eq:eqnSysCorr1} with $(\dy;\dx)^\top$ from the left demonstrates that every solution to \eqref{eq:eqnSysCorr} satisfies
\begin{equation}\label{eq:dxds}
\dx^\top\ds=0.
\end{equation}
The next lemma summarizes useful results from \cite{SY2015} about the corrector update direction.

\begin{lemma}[see (44-45) in \cite{SY2015}]\label{lem:dxdsBndCorr}
Let $z=(\bx;y;\bs)\in\cN(\theta)$ for some $0\leq\theta\leq 1$, and let $\dz=(\dx;\dy;\ds)$ be the solution to \eqref{eq:eqnSysCorr}. Then
\begin{equation}\label{eq:dxdsBndCorr}
\|\dx\|_{\bx}\leq\theta\quad\text{and}\quad\|\ds\|_{\bx}^*\leq\theta\mu(z).
\end{equation}
If in addition $0\leq\af_c\leq 1$, then
\begin{equation}\label{eq:psiWithNewSBnd}
\|\psi(z)+\af_c\ds\|_{\bx}^*\leq\theta\mu(z).
\end{equation}
\end{lemma}

\revise{
The next lemma shows how the linear residuals and the complementarity gap change as a result of a corrector step.

\begin{lemma}\label{lem:resRedCorr}
Let $z=(\bx;y;\bs)\in\cK^\circ$, and let $\dz=(\dx;\dy;\ds)$ be the solution to \eqref{eq:eqnSysCorr}. Choose $z^+=(\bx^+;y^+;\bs^+)=z+\af_c\dz$ for some scalar $\af_c$. Then
\begin{align}
G(y^+;\bx^+)-(0;\bs^+)&=G(y;\bx)-(0;\bs)\qquad\text{ and }\label{eq:infRedCorr}\\
\mu(z^+)&=\mu(z)-\af_c\bnu^{-1}\mu\|\dx\|_{\bx}^2.\label{eq:muRedCorr}
\end{align}
\end{lemma}
\begin{proof}
Equation \eqref{eq:infRedCorr} follows immediately from \eqref{eq:eqnSysCorr1}. To prove \eqref{eq:muRedCorr}, let $\mu:=\mu(z)$ and $\mu^+:=\mu(z^+)$. Rearranging \eqref{eq:dxds}, we get
\begin{equation}\label{eq:muChgCorrAux1}
\dx^\top\bs\overset{\eqref{eq:dxds}}{=} \dx^\top\bs+\dx^\top\ds \overset{\eqref{eq:eqnSysCorr2}}{=} \dx^\top\bs-\dx^\top(\psi + \mu\bH(\bx)\dx)= - \mu\left(\dx^\top\bg(\bx) + \|\dx\|_{\bx}^2\right).
\end{equation}
Furthermore, Theorem~\ref{thm:LHBprops} gives
\begin{align}\label{eq:muChgCorrAux2}
\bx^\top\ds \overset{\eqref{eq:eqnSysCorr2}}{=} -\bx^\top(\mu\bH(\bx)\dx + \psi) \overset{\eqref{eq:psix}}{=} - \mu\dx^\top(\bH(\bx)\bx) \overset{\eqref{eq:LHBprops}}{=} \mu\dx^\top\bg(\bx).
\end{align}
Now using \eqref{eq:muChgCorrAux1} and \eqref{eq:muChgCorrAux2}, we obtain
\begin{align*}
\mu^+-\mu
&=\bnu^{-1}\left(\bx^+{}^\top\bs^+ - \bx^\top\bs\right)
\overset{\eqref{eq:dxds}}{=}\af_c\bnu^{-1}\left(\dx^\top\bs + \bx^\top\ds\right)\notag\\
&\hspace{-1em}\overset{(\ref{eq:muChgCorrAux1}-\ref{eq:muChgCorrAux2})}{=}\af_c\bnu^{-1}\mu\left(-\dx^\top\bg(\bx) - \|\dx\|_{\bx}^2 + \dx^\top\bg(\bx)\right) = -\af_c\bnu^{-1}\mu\|\dx\|_{\bx}^2.
\end{align*}
This completes the proof of \eqref{eq:muRedCorr}.
\end{proof}

In particular, equation~\eqref{eq:muRedCorr} shows that the complementarity gap never increases during the corrector phase.

Using \eqref{eq:dxdsBndCorr} and \eqref{eq:muRedCorr}, one can further bound the change in the complementarity gap. The left-hand side inequality in \eqref{eq:muRatioCorr} can also be found in Appendix~4 of \cite{SY2015}.
}


\begin{lemma}[see also Appendix 4 in \cite{SY2015}]\label{lem:muChgCorr}
Let $z=(\bx;y;\bs)\in\cN(\theta)$ for some $0\leq\theta<1$, and let $\dz=(\dx;\dy;\ds)$ be the solution to \eqref{eq:eqnSysCorr}. Choose $z^+=(\bx^+;y^+;\bs^+):=z+\af_c\dz$ for some $0\leq\af_c\leq 1$. Then
\begin{gather}
|\mu(z^+)-\mu(z)|\leq\af_c\bnu^{-1}\theta^2\mu(z),\label{eq:muDiffCorr}\\
\revise{1-\af_c\bnu^{-1}\theta^2\leq\frac{\mu(z^+)}{\mu(z)}\leq 1.}
\label{eq:muRatioCorr}
\end{gather}
\end{lemma}

The analysis in \cite{SY2015} does not provide any details why the solution after a corrector step belongs to $\cF^\circ$. We provide a proof of this for completeness:

\begin{lemma}\label{lem:inCorr}
Let $z=(\bx;y;\bs)\in\cN(\theta)$ for some $0\leq\theta<1$, and let $\dz=(\dx;\dy;\ds)$ be the solution to \eqref{eq:eqnSysCorr}. Choose $z^+=(\bx^+;y^+;\bs^+):=z+\af_c\dz$ for some $0\leq\af_c\leq 1$. Then $z^+\in\cF^\circ$.
\end{lemma}
\begin{proof}
To see $\bx^+\in\bK^\circ$, note using \eqref{eq:dxdsBndCorr} that $\|\bx^+-\bx\|_{\bx}=\af_c\|\dx\|_{\bx}\leq\af_c\theta<1$. The result now follows from $\bx\in\bK^\circ$ and the self-concordance of $\bF$. To prove $\bs^+\in(\bK^*)^\circ$, let $\mu:=\mu(z)$, $\mu^+:=\mu(z^+)$, and $\psi^+:=\psi(\bx^+,\bs^+,\mu^+)=\bs^+ +\mu^+\bg(\bx^+)$. Recall from Theorem~\ref{thm:SCsymm}(i) that $\bx\in\bK^\circ$ implies $-\bg(\bx)\in(\bK^*)^\circ$. Furthermore, $\mu>0$ because $\bx\in\bK^\circ$ and $\bs\in(\bK^*)^\circ$. Using this, we get
\[
\|\bs^+/\mu + \bg(\bx)\|_{-\bg(\bx)}^* \overset{\eqref{eq:dNorms}}{=} \|\bs^+/\mu + \bg(\bx)\|_{\bx}^*
= \mu^{-1}\|\psi+\af_c \ds\|_{\bx}^* \overset{\eqref{eq:psiWithNewSBnd}}{\leq}\theta<1.
\]
The conclusion $\bs^+\in(\bK^*)^\circ$ now follows from the last inequality together with $-\bg(\bx)\in(\bK^*)^\circ$ and the self-concordance of $\bF^*$.
\end{proof}

The next proposition is the main result of this section.

\begin{proposition}\label{prop:psiBndCorr}
Let $z=(\bx;y;\bs)\in\cN(\theta)$ for some $0\leq\theta<1$, and let $\dz=(\dx;\dy;\ds)$ be the solution to \eqref{eq:eqnSysCorr}. Choose $z^+=(\bx^+;y^+;\bs^+):=z+\af_c\dz$ for some $0\leq\af_c\leq 1$. Then
\begin{equation}\label{eq:psiBndCorr}
\mu(z^+)^{-1}\|\psi(z^+)\|_{\bx^+}^*\leq(2-\af_c\theta^2)^{-1}\left(2\left(\frac{\af_c\theta}{1-\af_c\theta}\right)^2+\frac{4(1-\af_c)\theta}{1-\af_c\theta} +\sqrt{2}\af_c\theta^2\right).
\end{equation}
\end{proposition}
\begin{proof}
Let $\mu:=\mu(z)$, $\mu^+:=\mu(z^+)$, $\psi:=\psi(\bx,\bs,\mu)=\bs +\mu\bg(\bx)$, and $\psi^+:=\psi(\bx^+,\bs^+,\mu^+)=\bs^+ +\mu^+\bg(\bx^+)$. Note that $\mu>0$ because $\bx\in\bK^\circ$ and $\bs\in(\bK^*)^\circ$. Lemma~\ref{lem:inCorr} shows that $\bx^+\in\bK^\circ$ under the hypotheses of this proposition. We claim that
\begin{equation}\label{eq:nxpBnd}
\mu^{-1}\|\bs+\ds+\mu \bg(\bx^+)\|_{\bx^+}^*\leq
\left(\frac{\af_c\theta}{1-\af_c\theta}\right)^2+\frac{(1-\af_c)\theta}{1-\af_c\theta}.
\end{equation}
For the proof of the claim, consider the function $f:\bK^\circ\to\R$ defined as $f(v):=\mu^{-1}(\bs+\ds)^\top v + \bF(v)$. Note that $f$ is self-concordant and its Hessian is $\bH$. We denote the Newton step for $f$ at a point $v\in\bK^\circ$ with $n_f(v):=-\mu^{-1}\bH(v)^{-1}(\bs+\ds+\mu\bg(v))$. Note that
\begin{equation}\label{eq:nx}
n_f(\bx)=-\mu^{-1}\bH(\bx)^{-1}(\bs+\ds+\mu\bg(\bx))=-\mu^{-1}\bH(\bx)^{-1}(\psi+\ds)\overset{\eqref{eq:eqnSysCorr2}}=\dx,
\end{equation}
and therefore, $\bx+\af_c n_f(\bx)=\bx^+$. Furthermore, for any $v\in\bK^\circ$, we have
\begin{equation}\label{eq:nvNorm}
\|n_f(v)\|_v = \|\mu^{-1}\bH(v)^{-1}(\bs+\ds+\mu \bg(v))\|_v = \mu^{-1}\|\bs+\ds+\mu \bg(v)\|_v^*,
\end{equation}
and for $v=\bx$, we have
\begin{equation}\label{eq:nxNorm}
\|n_f(\bx)\|_{\bx}\overset{\eqref{eq:nx}}=\|\dx\|_{\bx}\overset{\eqref{eq:dxdsBndCorr}}\leq\theta.
\end{equation}
Given $\af_c\|n_f(\bx)\|_{\bx}\leq\af_c\theta<1$, we can conclude from Theorem~\ref{thm:renegarNew} that
\[
\mu^{-1}\|\bs+\ds+\mu \bg(\bx^+)\|_{\bx^+}^*
\overset{\eqref{eq:nvNorm}}{=}\|n_f(\bx^+)\|_{\bx^+}
\overset{\eqref{eq:renegarNew}}\leq
\left(\frac{\af_c\|n_f(\bx)\|_{\bx}}{1-\af_c\|n_f(\bx)\|_{\bx}}\right)^2
+\frac{(1-\af_c)\|n_f(\bx)\|_{\bx}}{1-\af_c\|n_f(\bx)\|_{\bx}}.
\]
Note that the right-hand side of this chain is monotone nondecreasing in $\|n_f(\bx)\|_{\bx}$ for $0\leq\af_c\leq 1$. Using \eqref{eq:nxNorm} to bound this right-hand side gives \eqref{eq:nxpBnd}.

To prove \eqref{eq:psiBndCorr}, we first use Theorem~\ref{thm:LHBprops} and the triangle inequality to obtain
\begin{align*}
(1/\mu^+)\|\psi^+\|_{\bx^+}^*
&= \frac{\mu}{\mu^+}\mu^{-1}\|\bs+\af_c \ds+\mu^+ \bg(\bx^+)\|_{\bx^+}^*\\
&\leq \frac{\mu}{\mu^+}\mu^{-1}\|\bs+\af_c \ds+\mu \bg(\bx^+)\|_{\bx^+}^* + \frac{|\mu^+-\mu|}{\mu^+}\|\bg(\bx^+)\|_{\bx^+}^*\\
&\leq \frac{\mu}{\mu^+}\mu^{-1}\left(\|\bs+ \ds+\mu \bg(\bx^+)\|_{\bx^+}^* +(1-\af_c)\| \ds\|_{\bx^+}^*\right)+ \frac{|\mu^+-\mu|}{\mu^+}\|\bg(\bx^+)\|_{\bx^+}^*\\
&\hspace{-0.5em}\overset{(\ref{eq:LHBprops},\ref{eq:nxpBnd})}{=}
\frac{\mu}{\mu^+}\left(\left(\frac{\af_c\theta}{1-\af_c\theta}\right)^2
+\frac{(1-\af_c)\theta}{1-\af_c\theta}+(1-\af_c)\mu^{-1}\|\ds\|_{\bx^+}^*\right) + \frac{|\mu^+-\mu|}{\mu^+}\sqrt{\bnu}.
\end{align*}
We can now use Lemmas~\ref{lem:dNormBaseChg}, \ref{lem:dxdsBndCorr}, and \ref{lem:muChgCorr} to bound the right-hand side of this chain and get
\begin{align*}
(1/\mu^+)\|\psi^+\|_{\bx^+}^*
&\hspace{-0.5em}\overset{(\ref{eq:dNormBaseChg},\ref{eq:dxdsBndCorr})}{\leq}
\frac{\mu}{\mu^+}\left(\left(\frac{\af_c\theta}{1-\af_c\theta}\right)^2
+\frac{(1-\af_c)(\theta+\mu^{-1}\|\ds\|_{\bx}^*)}{1-\af_c\theta}\right)
+ \frac{|\mu^+-\mu|}{\mu^+}\sqrt{\bnu}\\
&\overset{\eqref{eq:dxdsBndCorr}}{\leq}
\frac{\mu}{\mu^+}\left(\left(\frac{\af_c\theta}{1-\af_c\theta}\right)^2
+\frac{2(1-\af_c)\theta}{1-\af_c\theta}\right) + \frac{|\mu^+-\mu|}{\mu^+}\sqrt{\bnu}\\
&\hspace{-0.75em}\overset{(\ref{eq:muDiffCorr}-\ref{eq:muRatioCorr})}\leq
(1-\af_c\bnu^{-1}\theta^2)^{-1}\left(\left(\frac{\af_c\theta}{1-\af_c\theta}\right)^2+\frac{2(1-\af_c)\theta}{1-\af_c\theta} + \frac{\af_c\theta^2}{\sqrt{\bnu}}\right).
\end{align*}
Recall that $\bnu\geq 2$. Observing that the right-hand side of this chain is monotone nonincreasing in $\bnu$ for $\bnu\geq 2$ and replacing $\bnu$ with $2$ yields \eqref{eq:psiBndCorr}.
\end{proof}

Repeated application of the next corollary $r_c$ times implies Lemma~\ref{lem:newLemma6} stated at the beginning of Section~\ref{sec:algo}.

\begin{corollary}
Assume one of the following:
\begin{itemize}
\item $\theta\leq 0.20$, $\epsilon=0.50$, and $\af_c=1$, or
\item $\theta\leq 0.25$, $\epsilon=0.70$, and $\af_c=1$.
\end{itemize}
Let $z=(\bx;y;\bs)\in\cN(\theta)$, and let $\dz=(\dx;\dy;\ds)$ be the solution to \eqref{eq:eqnSysCorr}. Then $z^+=(\bx^+;y^+;\bs^+):=z+\af_c\dz\in\cN(\epsilon\theta)$.
\end{corollary}
\begin{proof}
It can be verified from Lemma~\ref{lem:inCorr} and Proposition~\ref{prop:psiBndCorr} that $z^+\in\cF^\circ$ and $\|\psi(\bx^+,\bs^+,\mu(z^+))\|_{\bx^+}^*\allowbreak\leq\epsilon\theta\mu(z^+)$ with the prescribed parameters.
\end{proof}

\subsection{Complexity}

With the above analysis of the predictor and the corrector steps, establishing the iteration complexity of the algorithm (stated precisely as Theorem 1 in \cite{SY2015}) is straightforward; we provide a sketch of the argument for completeness, following Appendix 5 in \cite{SY2015}.

\revise{
It is easy to show via Lemmas~\ref{lem:resRedPred} and \ref{lem:resRedCorr} that the feasibility gap $\|G(y;\bx)-(0;\bs)\|$ and the complementarity gap $\mu(z)$ decrease geometrically. Using a predictor step size $\af_p$, the feasibility gap decreases by a factor of $(1-\af_p)$ in each predictor step (see \eqref{eq:infRedPred}) and does not change in the corrector steps (see \eqref{eq:infRedCorr}). Similarly, the complementarity gap $\mu(z)$ decreases by a factor of $\left(1-\af_p\right)\left(1-\af_p\eta\kx\bnu^{-1}\right)$ in each predictor step (see \eqref{eq:muRatioPred}) and does not increase in the corrector steps (see \eqref{eq:muRatioCorr}).

With the parameter choices suggested in Lemma~\ref{lem:newLemma5}, the predictor step size $\af_p$ is of order $\Omega(1/\sqrt{\nu})$, and hence both the feasibility gap and the complementarity gap are reduced by a factor of $O(1-1/\sqrt{\nu})$ in each iteration. As a result, both gaps are reduced to $\varepsilon$ times their initial values in $O(\sqrt{\nu}\log(1/\varepsilon))$ iterations.
}

\bibliographystyle{abbrv}
\bibliography{ref}

\begin{thebibliography}{1}

\bibitem{N2012}
Y.~Nesterov.
\newblock Towards non-symmetric conic optimization.
\newblock {\em Optimization Methods \& Software}, 27(4-5):893--917, 2012.

\bibitem{NN1994}
Y.~Nesterov and A.~Nemirovski.
\newblock {\em Interior-Point Polynomial Algorithms in Convex Programming}.
\newblock Studies in Applied and Numerical Mathematics. SIAM, Philadelphia, PA,
  1994.

\bibitem{NTY1999}
Y.~Nesterov, M.~J. Todd, and Y.~Ye.
\newblock Infeasible-start primal-dual methods and infeasibility detectors for
  nonlinear programming problems.
\newblock {\em Mathematical Programming}, 84:227--267, 1999.

\bibitem{Ren2001}
J.~Renegar.
\newblock {\em A Mathematical View of Interior-Point Methods in Convex
  Optimization}.
\newblock Number MP03 in {MPS/SIAM} Series on Optimization. SIAM, Philadelphia,
  PA, 2001.

\bibitem{SY2015}
A.~Skajaa and Y.~Ye.
\newblock A homogeneous interior-point algorithm for nonsymmetric convex conic
  optimization.
\newblock {\em Mathematical Programming Ser. A}, 150:391--422, 2015.

\end{thebibliography}

\appendix

\section{Omitted Proofs}

\begin{proof}[Proof of Theorem~\ref{thm:renegarNew}]
First note that
\begin{align}
\|n(x^+)\|_{x^+}^2
&=\|H(x^+)^{-1}g(x^+)\|_{x^+}^2\notag\\
&=\la g(x^+),H(x)^{-1}H(x^+)^{-1}g(x^+)\ra_{H(x)}\notag\\
&=\la H(x)^{-1}g(x^+),H(x^+)^{-1}H(x)H(x)^{-1}g(x^+)\ra_{H(x)}\notag\\
&\overset{\eqref{eq:lemma1ii}}{\leq}\|H(x^+)^{-1}H(x)\|_x\|H(x)^{-1}g(x^+)\|_x^2\notag\\
&=\|H(x^+)^{-1}H(x)\|_x(\|g(x^+)\|_x^*)^2.\label{eq:renegarNewPrf1}
\end{align}
Given $\|x^+-x\|_x=\af\|n(x)\|_x<1$, we conclude from Theorem~\ref{thm:SCequi} that
\[
\|H(x^+)^{-1}H(x)\|_x\overset{\eqref{eq:SCequi1}}{\leq}\frac{1}{(1-\af\|n(x)\|_x)^2}.
\]
Combining this with \eqref{eq:renegarNewPrf1} yields
\begin{equation}\label{eq:renegarNewPart1}
\|n(x^+)\|_{x^+}\leq\frac{\|g(x^+)\|_x^*}{1-\af\|n(x)\|_x}.
\end{equation}
To finish the proof, we prove an upper bound on $\|g(x^+)\|_x^*$. Our proposition will follow from this upper bound and the inequality \eqref{eq:renegarNewPart1} above. Note that the fundamental theorem of calculus for the gradient (see Proposition 1.5.7 in \cite{Ren2001}) implies
\[
g(x^+)-g(x)-\af H(x)n(x)=\int_0^1\left(H(x+t\af n(x))-H(x)\right)\af n(x)dt.
\]
Using this together with $n(x)=-H(x)^{-1}g(x)$ and the triangle inequality, we get
\begin{align}
\|g(x^+)\|_x^*&=\|g(x^+)-g(x)-H(x)n(x)\|_x^*\notag\\
&\leq\|g(x^+)-g(x)-\af H(x)n(x)\|_x^*+(1-\af)\|H(x)n(x)\|_x^*\notag\\
&=\left\|\int_0^1(H(x+t\af n(x))-H(x))\af n(x)dt\right\|_x^*+(1-\af)\|H(x)n(x)\|_x^*\notag\\
&\leq\int_0^1\left\|\big(H(x+t\af n(x))-H(x)\big)\af n(x)\right\|_x^* dt+(1-\af)\|H(x)n(x)\|_x^*\notag\\
&=\int_0^1\left\|\big(H(x)^{-1}H(x+t\af n(x))-I\big)\af n(x)\right\|_x dt+(1-\af)\|n(x)\|_x.\label{eq:renegarNewPrf2}
\end{align}
For every $0\leq t\leq 1$, Theorem~\ref{thm:SCequi} indicates
\begin{align}
\left\|\big(H(x)^{-1}H(x+t\af n(x))-I\big)\af n(x)\right\|_x&\overset{\eqref{eq:lemma1i}}{\leq} \af\|n(x)\|_x\|H(x)^{-1}H(x+t\af n(x))-I\|_x\notag\\
&\overset{\eqref{eq:SCequi2}}{\leq}\af\|n(x)\|_x\left(\frac{1}{(1-t\af\|n(x)\|_x)^2}-1\right).\label{eq:renegarNewPrf3}
\end{align}
Putting \eqref{eq:renegarNewPrf2} and \eqref{eq:renegarNewPrf3} together, we obtain
\begin{align*}
\|g(x^+)\|_x^*&
\overset{\eqref{eq:renegarNewPrf2}}{\leq}\int_0^1\|(H(x)^{-1}H(x+t\af n(x))-I)\af n(x)\|_x dt+(1-\af)\|n(x)\|_x\\
&\overset{\eqref{eq:renegarNewPrf3}}{\leq}\af\|n(x)\|_x\int_0^1\left(\frac{1}{(1-t\af\|n(x)\|_x)^2}-1\right)dt+(1-\af)\|n(x)\|_x\\
&\leq\frac{\af^2\|n(x)\|_x^2}{1-\af\|n(x)\|_x}+(1-\af)\|n(x)\|_x.
\end{align*}
The last inequality together with \eqref{eq:renegarNewPart1} yields our theorem.
\end{proof}

\section{Counterexamples}

In this appendix, we consider the following primal-dual pair of linear programming problems:
\vspace{-1.5em}
\[
\begin{minipage}[t]{0.45\linewidth}\centering
\begin{alignat*}{2}
\text{(LP-P):}\;    &\min_x\;           &&2x_1+3x_2\\
                    &\;\text{s.t.}\;    &&5x_1-3x_2=12\\
                    &                   &&x\in\R^2_+
\end{alignat*}
\end{minipage}
\begin{minipage}[t]{0.45\linewidth}\centering
\begin{alignat*}{2}
\text{(LP-D):}\;    &\max_{s,y}\;       &&12y\\
                    &\;\text{s.t.}\;    &&5y+s_1=2\\
                    &                   &-&3y+s_2=3\\
                    &                   &&s\in\R^2_+,\; y\in\R
\end{alignat*}
\end{minipage}
\]
In the notation of this paper, we have $A=(5,-3)$, $c=(2;3)$, and $b=12$. The standard logarithmic barrier $F(x)=-\log x_1-\log x_2$ for $K=\R^2_+$ has barrier parameter $\nu=2$. We assume $\beta=0.30$ and $\eta=0.50\beta$ in the remainder of this appendix. We also set $\bnu=\nu+1$, $\kx=\eta+\sqrt{2\eta^2+\bnu}$, and $\ks=\kx+\sqrt{\kx^2+\eta^2+\bnu}$ as described in Section~\ref{sec:algo}.\footnote{In equation (34) of \cite{SY2015}, Skajaa and Ye have $\kx=\eta+\sqrt{\eta^2+\bnu}$ in contrast with the definition used in this note. However, given the developments that precede the definition of $\kx$ in \cite{SY2015}, we believe that this difference is due to a typographical error. The counterexamples stated in this appendix remain valid for both definitions.} Let $\af_p=0.052$ and $\af_c=1/84$. This $\af_p$ value satisfies all of the three conditions stated in Appendix 3 of \cite{SY2015}:
\[
\af_p\leq\kx^{-1},\qquad \af_p\leq (1-\eta)\ks^{-1},\quad\text{and}\quad \af_p\leq (11\sqrt{\bnu})^{-1}.
\]
Furthermore, the $\af_c$ value satisfies the condition $\af_c\leq 1/84$ stated in Appendix 4 of \cite{SY2015}.

We first consider the proof of Lemma~5 presented in \cite{SY2015}. Let $z=(\bx;y;\bs)\in\cN(\eta)$, let $\dz=(\dx;\dy;\ds)$ be the corresponding solution to \eqref{eq:eqnSysPred}, and define $z^+=(\bx^+;y^+;\bs^+):=z+\af_p\dz$. Then inequality (37) of \cite{SY2015} claims that if $q:=\mu(z^+)^{-1}\|\bs^+ +\mu(z^+)\bg(\bx)\|_{\bx}^*<1$, then
\begin{equation}\label{eq:3747}
\|\psi(\bx^+,\bs^+,\mu(z^+))\|_{\bx^+}^*\leq\frac{\mu(z^+)q^2}{(1-q)^2}.
\end{equation}
However, this inequality clearly does not hold for the solution $x=(0.9310;0.6995)$, $\tau=0.8511$, $y=0.0224$, $s=(0.8246;1.0891)$, and $\kappa=0.9023$, which satisfies $z\in\cN(\eta)$. The main problem in the proof is the incorrect use of Theorem~\ref{thm:renegarOrg}: here $\bx^+$ is the result of a dampened Newton step, while Theorem~\ref{thm:renegarOrg} applies to a full Newton step.

Now we consider the proof of Lemma~6 from \cite{SY2015}. Let $z=(\bx;y;\bs)\in\cN(\beta)$, let $\dz=(\dx;\dy;\ds)$ be the corresponding solution to \eqref{eq:eqnSysCorr}, and define $z^+=(\bx^+;y^+;\bs^+):=z+\af_c\dz$. Then inequality (47) of \cite{SY2015} claims that a similar result to the one above is true for all such choices of $z$: If $q<1$, then \eqref{eq:3747} must be satisfied. However, this inequality does not hold for the solution $x=(0.9830;0.9304)$, $\tau=0.9670$, $y=0.0042$, $s=(0.9650;1.0176)$, and $\kappa=0.9810$, which satisfies $z\in\cN(\beta)$. The main problem is again the incorrect use of Theorem~\ref{thm:renegarOrg}.

\end{document}